\documentclass[12pt]{amsart}
\usepackage{amsfonts,amssymb,mathrsfs, amsthm, textcomp,amscd,amsbsy,amsmath}
\usepackage[small,nohug,heads=littlevee]{diagrams}
\diagramstyle[labelstyle=\scriptstyle]

\usepackage{verbatim}
\usepackage{enumerate}
\usepackage{amsmath}
\usepackage{geometry}
\usepackage[latin5]{inputenc}
\usepackage{yfonts}

   \newtheorem{theorem}{Theorem}[section]
   \newtheorem{lemma}{Lemma}[section]

\newtheorem{rem}{Remark}[section]
\newtheorem{prop}{Proposition}[section]
\newtheorem{coro}{Corollary}[section]

\newcommand{\Aut}{\mathrm{Aut}}

\newcommand{\Def}{\mathrm{Def}}

\newcommand{\Hom}{\mathrm{Hom}}

\newcommand{\ind}{\mathrm{Ind}}
\newcommand{\Ind}{\mathrm{Ind}}
\newcommand{\infl}{\mathrm{Inf}}
\newcommand{\Inf}{\mathrm{Inf}}

\newcommand{\lexp}[2]{\setbox0=\hbox{$#2$} \setbox1=\vbox to
                 \ht0{}\,\box1^{#1}\!#2}

\newcommand{\res}{\mathrm{Res}}
\newcommand{\Res}{\mathrm{Res}}
\newcommand{\Out}{\mathrm{Out}}
\newcommand{\out}{\mathrm{Out}}

\newcommand{\tw}{\mathrm{Tw}}

\title{Fibered $p$-biset functor structure of the fibered Burnside rings}
\author{Olcay Coşkun \and Deniz Yılmaz}

\address{Boğaziçi University Department of Mathematics Bebek İstanbul Turkey} 

\address{University of California, Santa Cruz Department of Mathematics CA 95064 USA}
\thanks{
Both of the authors are supported by Tübitak-1001-113F40.}

\begin{document}

\maketitle

\begin{abstract}
We determine the composition factors of the $A$-fibered Burnside functor $kB^{A}$ for $p$-groups over a field $k$ of characteristic $q$ with $q\neq p$ and cyclic fiber group $A$. We also show that, in this case, $kB^{A}$ is uniserial. 
\end{abstract}
\section{Introduction}

In representation theory, it is of the utmost importance to study group actions on sets. For the simplest case, one may 
consider the action of a finite group $G$ on a finite set $X$. This action reveals the theory of Burnside rings. Considering the common features shared by Burnside rings and representation 
rings, Dress \cite{dressmackey} and Green  \cite{green} introduced Mackey functors to give a unified treatment of these 
objects. 

There are two ways to let two groups act on a set $X$.
First, suppose that we have two groups $G$ and $H$. By considering the action of $G$ on the left and the action of $H$ on the right, we may let $G\times H$ act on $X$.
In this case, the set $X$ is called a $(G,H)$-biset and this leads us to the theory of biset functors introduced by Bouc in \cite{boucdecomposition}. One of the most important applications of biset functors, among many others, is the final determination of the structure of the Dade group by Bouc \cite{BDade}. Also in \cite{boucthevenaz}, Bouc and Thev\'{e}naz studied the Burnside functor of $p$-groups. They obtained that the Burnside functor of $p$-groups over a field of characteristic zero has two composition factors, one of which is the functor of torsion-free part of the Dade group and the other one is the functor of rational representations.

As a second way of letting two groups act on a set, we may consider the action of $A\times G$ on $X$ where $G$ is a finite group and $A$ is an abelian group acting on $X$ freely. 
Since the $A$-action is free, such an action of $A\times G$ on $X$ can be considered as $G$ acting on the $A$-fibers and in this case, the set $X$ is called an $A$-fibered $G$-set.
These objects were introduced by Dress in \cite{dress} and studied by Boltje \cite{boltje} and Barker \cite{barker}.

In \cite{boltjecoskun}, Boltje and the first author combined these two notions and introduced $A$-fibered $(G,H)$-bisets.  Our 
aim in the present paper is to extend the results of Bouc and Thev\'{e}naz on Burnside functors and to determine the composition 
factors of the $A$-fibered Burnside functor $kB^{A}$ of $p$-groups over a field $k$ of characteristic 
$q$ with $q\neq p$ and a cylic fiber group $A$. To be more precise, we show that in this case, the functor $kB^{A}$ is uniserial with composition factors
parameterized by sets of elementary abelian $p$-groups depending only on the prime $p$ and the characteristic $q$, and not on the particular fiber 
group $A$. 

As in the case of the (ordinary) biset functor of the Burnside functor, when $q=0$, the fibered Burnside functor $kB^{A}$ has only two composition 
factors. One of the factors can be identified with a subfunctor of the functor of complex characters. The other factor has the cyclic group $C_{p}$ as its 
minimal group, but we are unable to identify it with a natural construction. 

In Section \ref{prel}, we review the theory of fibered bisets and fibered biset functors from \cite{boltjecoskun}, together with some specializations to the
case of abelian groups. In the next section, we introduce fibered Burnside rings, recall the idempotent formula for its primitive idempotents form 
\cite{barker} and determine the action of basic fibered bisets on these idempotents. Our main results regarding the structure of the fibered Burnside 
functors are contained in the last section.

\section{Fibered Bisets and Fibered Biset Functors}\label{prel}
In this section, we recall basic theory of fibered bisets from \cite{boltjecoskun} and specialize certain results to the case of abelian groups with
sufficiently large fiber groups.
\subsection{Fibered bisets}
Let $G$ be a finite group, $A$ be an abelian group and $X$ be a set. We call $X$ an \textit{$A$-fibered $G$-set} if $X$ is an $A
\times G$-set such that the action of $A$ is free with finitely many orbits. We denote by ${}_G\mathrm{set}^A$
the category of $A$-fibered $G$-sets. Here the morphisms are given by $A\times G$-equivariant functions. The 
operation of disjoint union of sets induces a coproduct on ${}_G\mathrm{set}^A$ and  we denote by $B^A(G)$ the 
Grothendieck group of this category with respect to disjoint unions. The group $B^A(G)$ is called the 
\textit{$A$-fibered Burnside group} and it was first introduced, in a more general way, by Dress in \cite{dress}.

The basic objects in ${}_G{\mbox{\rm set}}^A$ with respect to disjoint union are the transitive ones. We say that an 
$A$-fibered $G$-set is \textit{transitive} provided that 
the $G$-action on the set of $A$-orbits is transitive. It is easy to show that there is a bijective correspondence
between the isomorphism classes $[X]$ of transitive $A$-fibered $G$-sets and the $G$-conjugacy classes of 
pairs $(U,\phi)$ where $U$ is a subgroup of $G$ and $\phi: U\rightarrow A$ is a group homomorphism. The 
bijection is given by associating $X$ to $(U, \phi)$ if $U$ is the stabilizer of some $A$-orbit in $X$ and $U$ acts on 
this $A$-free orbit via $\phi$. We call the pair $(U,\phi)$ corresponding to $X$ the \textit{stabilizing pair} of $X$. 
We denote by $\mathcal M_G(A)$ the set of all such pairs $(U, \phi)$, and write $[U,\phi]_G$ for the isomorphism class 
of the $A$-fibered $G$-set with the stabilizing pair $(U,\phi)$. 

If $H$ is another finite group, we write ${}_G\mathrm{set}_H^A$ for the category of $A$-fibered $G\times H$-sets. 
By the usual convention, we regard any object in this category as an $A$-fibered $(G,H)$-biset. We also regard
any $A$-fibered biset as an operator and, in this case, we write $\big[\frac{G\times H}{U,\phi}\big]$ instead of 
$[U,\phi]_{G\times H}$. With this notation, any ordinary biset $\big[\frac{G\times H}{U} \big]$ is regarded as an $A$-fibered biset as 
$\big[\frac{G\times H}{U,1} \big]$ where $1$ denotes the trivial homomorphism sending any element of $U$ to the identity element
of $A$.

Further let $K$ be another finite group, $X$ an $A$-fibered $(G,H)$-biset and $Y$ an $A$-fibered $(H,K)$-biset. 
We define the tensor product $X\otimes_{AH} Y$ of $X$ and $Y$ as the set of $A$-free orbits of the usual amalgamated product
$X\times_{AH} Y$ of the bisets $X$ and $Y$. Recall that $X\times_{AH} Y$ is the set of $A\times H$-
orbits in $X\times Y$ under the $A\times H$-action given by
\[
(a,h)\cdot (x,y) = (x\cdot (a^{-1},h^{-1}),(a,h)\cdot y)
\]
for any $(a,h)\in A\times H$ and $(x,y)\in X\times Y$. We denote the $A\times H$-orbit containing the pair $(x, y)$ by $(x,_{AH} y)$. Then the group $A$ acts 
on the set $X\times_{AH} Y$ via $a\cdot (x,_{AH} y) = (a\cdot x,_{AH} y) = (x,_{AH} a\cdot y)$. Given an $A$-free $A\times H$-orbit $(x,_{AH} y)$, we denote
its image in the subset $X\otimes_{AH} Y$ by $x\otimes_{AH} y$ or simply by $x\otimes y$, when there is no risk of confusion. With this notation, $X\otimes_{AH} Y$ becomes an $A$-fibered $(G,K)$-biset via
\[
(g,a)\cdot (x\otimes y) \cdot k = (g\cdot a \cdot x) \otimes (y\cdot k)
\]
for $g\in G, a\in A$ and $k\in K$.
We introduce further notation to determine the product of two transitive $A$-fibered bisets.

Given a pair $(U,\phi)\in \mathcal M_{G\times H}(A)$, the subgroup $U$ determines the following datum:
Let $P = p_1(U)$ and $Q= p_2(U)$ be the first and the second 
projections of $U$. Let also $K = k_1(U) = p_1(U\cap (G\times 1))$ and $L = k_2(U) = p_2(U\cap (1\times H))$. Then
we have that $K\unlhd P$ and $L\unlhd Q$. Moreover the groups $P/K$ and $Q/L$ are isomorphic and a canonical 
isomorphism $\eta: Q/L\to P/K$ is determined by the subgroup $U$ via $\eta(hL) = gK$ if $(g,h)\in U$. Conversely, if a 
quintuple $(P,K,\eta, L,Q)$ where $K\unlhd P\le G$ and $L\unlhd Q\le H$ and $\eta:Q/L\to P/K$ an isomorphism
is given, a subgroup $U$ with the given invariants is uniquely determined by $U = \{(g,h)\in P\times Q \mid \eta(hL) =
 \eta(gK)\}$. This is known as Goursat's Theorem. We further write $\phi\vert_{K\times L} = \phi_1\times\phi_2^{-1}$.

With this notation, if both $X$ and $Y$ are transitive, say $X = \big[\frac{G\times H}{U,\phi}\big]$ and 
$Y= \big[\frac{H\times K}{V,\psi}\big]$, then by \cite[Corollary 2.5]{boltjecoskun}, the above tensor product becomes
\begin{equation}\label{product:transitives}
\Big[\frac{G\times H}{U,\phi}\Big]\otimes_{AH} \Big[\frac{H\times K}{V,\psi}\Big] = \sum_{\substack{x\in [p_2(U)
\backslash H/p_1(V)]\\ \phi_2\vert_{H_x} = \lexp{x}\psi_{1}\vert_{H_x}}} \Big[\frac{G\times K}{U\ast \lexp{(x,1)}V,\phi\ast\lexp{(x,1)}\psi}\Big]
\end{equation}
where $H_x = k_2(U)\cap\lexp{x}k_1(V)$, the subgroup $U\ast V$ is the composition 
\[
U\ast V = \{ (g,k)\in G\times K\vert (g,h)\in U, (h,k)\in V \, \mbox{\rm for some}\, h\in H \}
\]
and the homomorphism $\phi\ast \psi: U\ast V\to A$ is defined by 
\[
(\phi\ast\psi)(g,k) = \phi(g,h)\cdot\psi(h,k)
\]
for some choice of $h\in H$ such that $(g,h)\in U$ and $(h,k)\in V$. Note that the homomorphism $\phi\ast\psi$ is independent of 
the choice of $h\in H$. The equation $(\ref{product:transitives})$ is sometimes referred as the \textit{Mackey product formula}.

\subsection{Decompositions for abelian groups}
This product allows us to decompose any $A$-fibered $(G,H)$-biset into basic ones, as in the case of ordinary bisets. We refer to \cite{boltjecoskun} for 
further details. In this paper, we only need the decomposition of fibered bisets for abelian groups with sufficiently large fiber group $A$, which we 
discuss next. First, we introduce the notation for basic fibered bisets which is used throughout the paper. Let $H$ be a subgroup of $G$ and $N$ be a 
normal subgroup of $G$. Also let $G^\prime$ be another finite  group with a group isomorphism $\lambda: G\to G^\prime$. 
  
Following Bouc \cite{bouc}, we define the \emph{induction} from $H$ to $G$ and \emph{restriction} from $G$ to $H$ as the transitive bisets
\[
\ind_H^G := {}_GG_H,\quad \res^G_H := {}_HG_G
\]
where we regard the set $G$ as a $(G,H)$-biset (resp. as an $(H,G)$-biset) in the usual way, via left and right 
multiplication by the corresponding group. We also define \emph{deflation} from $G$ to $G/N$ and \emph{inflation} from $G/N$ 
to $G$ as the transitive bisets
\[
\Def_{G/N}^G := {}_{G/N}(G/N)_G,\quad \infl_{G/N}^G:= {}_G(G/N)_{G/N}.
\]
As above, we regard the set $G/N$ as a $(G/N,G)$-biset (and as a $(G,G/N)$-biset) in the usual way. 
Finally, we 
define the \emph{transport of structure} from $G^\prime$ to $G$ through $\lambda$ as the biset
\[
\mbox{\rm c}_{G,G^\prime}^\lambda:= {}_GG_{G^\prime}
\]
where the $G$-action is the left multiplication and the $G^\prime$-action is multiplication through $\lambda$.
In all these cases, the $A$-action is trivial.

Another basic fibered biset that we need in this paper is the twist biset defined as follows. Let 
$\phi\in G^* = \Hom(G,A)$ be a homomorphism from $G$ to $A$. Then the \textit{twist} by $\phi$ at $G$ is the $A$-fibered 
$(G,G)$-biset  
$$\tw_G^\phi=\Big( \frac{G\times G}{\Delta(G), \Delta(\phi)}\Big).$$
Here, for any pair $(G, \phi)\in \mathcal M_G(A)$, the diagonal inclusion $(\Delta(G),\Delta(\phi))$ of $(G,\phi)$ in 
$\mathcal M_{G\times G}(A)$ is the pair consisting of the diagonal inclusion of the group $G$ in $G\times G$ and
the diagonal homomorphism $\Delta(\phi)$ given by $\Delta(\phi)(g, g) = \phi(g)$. 

Now let $G$ be a finite abelian group and $A$ be an abelian group. We say that $A$ is \emph{splitting for $G$} if $A$ contains an 
element of order exp$(G)$. Note that, in this case, homomorphisms $G\to A$ can be identified with homomorphisms $G\to \mathbb C$.
For the next theorem, we let $G$ and $H$ be abelian groups and $A$ be splitting for both $G$ and $H$.  
Then given a pair $(U, \phi)\in \mathcal M_{G\times H}(A)$,  we write $(P,K,\eta,Q,L)$ for the invariants 
$(p_1(U),k_1(U), \eta,p_2(U),k_2(U))$ where $\eta$ is the canonical isomorphism between $P/K$ and $Q/L$ determined by $U$. We also write 
$\tilde \phi = \tilde\phi_1\times \tilde\phi_2$ for an extension of $\phi$ to $P\times Q$ which exists by the above assumption on $A$.
\begin{theorem}\label{thm:decomposition}
Assume the above notation. Then there is an isomorphism of $A$-fibered $(G,H)$-bisets
\[
\Big(\frac{G\times H}{U,\phi}\Big) \cong \Ind_{P}^G \tw_{P}^{\tilde \phi_1}\Inf_{P/K}^{P}
\mbox{\rm c}_{P/K,Q/L}^\eta\Def^{Q}_{Q/L}\tw^{\tilde \phi_2}_{Q}\Res^H_{Q}.
\]
\end{theorem} 

\begin{proof}
We evaluate the product on the right hand side. Note that, in the above product, the stabilizer for $\tw_P^{\tilde\phi_1}$ is $\Delta(P)$  and that for 
$\tw_Q^{\tilde\phi_2}$ is $\Delta(Q)$. Also, it is clear from its definition that for any subgroups $V\le G\times P$
and $V^\prime\le Q\times H$, we have $V\ast \Delta(P) = V$ and $\Delta(Q)\ast V^\prime = V^\prime$. Thus, the $\ast$-product of the 
stabilizers on the right hand side gives the subgroup $U$, as in the case of ordinary bisets given in Lemma 2.3.26 in \cite{bouc}. Thus we only need to check that 
$\phi = (1\ast(\tilde\phi_1 \ast 1))\ast ((1\ast\tilde \phi_2)\ast 1)$.
Let $(g,h)\in U$. Then we have
\begin{eqnarray*}
\big((1\ast(\tilde\phi_1 \ast 1))\ast ((1\ast\tilde \phi_2)\ast 1)\big)(g,h) &=& (1\ast(\tilde\phi_1 \ast 1))(g,gK)\cdot ((1\ast\tilde \phi_2)\ast 1)(hL,h)\\
&=& \tilde\phi_1(g)\cdot \tilde \phi_2(h) = (\tilde\phi_1 \times \tilde\phi_2)(g,h) = \phi(g,h)
\end{eqnarray*}
which completes the proof of the theorem. \qed
\end{proof}

\begin{rem} Suppose $A$ satisfies Hypothesis 10.1 in \cite{boltjecoskun}, so that tor$A$ is divisible. Then the above condition on $A$ is satisfied
trivially and hence the above theorem holds in this case. Note that \cite[Theorem 10.14]{boltjecoskun} describes the decomposition of an $A$-fibered
$(G,H)$-biset for any finite groups $G$ and $H$ with the fiber group $A$ satisfying the hypothesis.
\end{rem}

\begin{rem} Our main interest in this paper is the case where $A$ is a finite non-trivial cyclic $p$-group for a prime $p$ and both $G$ and $H$ are elementary abelian $p$-groups. Clearly, the above theorem also holds in this case.
\end{rem}


\subsection{Fibered Biset Functors} 
Let $A$ be an abelian group and $R$ be a commutative ring with unity. Let $\mathcal C:= \mathcal C_R^A$ denote 
the following category. The objects of $\mathcal C$ are all finite groups. Given two finite groups $G$ and $H$, we define
$$\Hom_{\mathcal C}(G,H) := R\otimes B(H\times G,A).$$
 The composition is the $R$-linear extension of the tensor product of $A$-fibered bisets introduced above.

Now an \emph{$A$-fibered biset functor over $R$} is an $R$-linear functor $\mathcal C\to {}_R$Mod. The class
of all $A$-fibered biset functors together with natural transformations between them form a category, denoted by
$\mathcal F:=\mathcal F_R^A$. Since ${}_R$Mod is an abelian category, the category $\mathcal F$ is also 
abelian. By the general theory for simple functors in such categories, see \cite[Section 2]{bouc}, simple fibered biset functors can be parameterized
by their evaluations at groups which are of minimal order having non-zero evaluation. Explicit classification of simple fibered biset functors over a field is 
done in \cite[Section 9]{boltjecoskun}. In this paper, we only need a special case of this parametrization, where the group $G$ is an elementary abelian 
$p$-group for a prime number $p$. It turns out that our techniques are also valid for functors parameterized by abelian groups. Next we consider this 
special case for completeness.
 

\subsection{Minimal evaluations of simple fibered biset functors for abelian groups}
In this section, we consider simple fibered biset functors with minimal non-zero evaluations at abelian groups. Let $G$ be a finite group and $F$ 
be a fibered biset functor. Also let $E_G$ denote the endomorphism ring of $G$ in $\mathcal C$. Clearly, the evaluation $F(G)$ is an $E_G$-module.
Furthermore, if $F$ is simple, then $F(G)$ is a module for the quotient algebra $\bar E_G =E_G/I_G$. Here $I_G$ denote the ideal generated by 
elements in $E_G$ which factor through a group of smaller order. The structure of $\bar E_G$ is described in \cite[Section 8]{boltjecoskun}. It turns out
that when $G$ is abelian, its structure is simpler, as we describe below. 

For the rest of this section, let $G$ be a finite abelian group and $A$ be an abelian group which is splitting for $G$. We first describe the structure of 
$\bar E_G$. Note that a similar result in the case where $A$ is cyclic of prime order and $G$ is arbitrary can be found in Lemma 15 of \cite{R}.

Let $G^*\rtimes \out(G)$ denote the semidirect product of the groups $G^*$ and $\out(G)$ where $\out(G)$ acts on $G^*$ via composition, that is,
$(\lambda\cdot \phi)(g) = \phi(\lambda(g))$ for any $g\in G, \phi\in G^*$ and $\lambda\in\out(G)$. With this notation, we have the following result.

\begin{theorem}
Let $G$ be a finite abelian group and $A$ be an abelian group which is splitting for $G$. Then there is an isomorphism of $R$-algebras
\[
\bar E_G \cong R[G^*\rtimes \Out(G)].
\]
\end{theorem}
\begin{proof}
Let $X = \big(\frac{G\times G}{U,\phi}\big)$ be a transitive $A$-fibered $(G,G)$-biset. With the notation of Theorem \ref{thm:decomposition}, we can write
\[
\Big(\frac{G\times H}{U,\phi}\Big) \cong \Ind_{P}^G \tw_{P}^{\tilde \phi_1}\Inf_{P/K}^{P}
\mbox{\rm c}_{P/K,Q/L}^\eta\Def^{Q}_{Q/L}\tw^{\tilde \phi_2}_{Q}\Res^H_{Q}.
\]
By this isomorphism, it is clear that $X$ is in the ideal $I_G$ unless $P = Q = G$ and $K = L = 1$. Therefore the quotient $\bar E_G$ can be identified with the 
submodule of $E_G$ generated by all $A$-fibered $(G,G)$-bisets $X = \big(\frac{G\times G}{U,\phi}\big)$ where
\begin{align*}
P = Q = G, \quad K=L=1,\quad U = \{ (g,\lambda(g))\in G\times G\vert \lambda\in\Aut(G)\}.
\end{align*}
Thus by Theorem \ref{thm:decomposition}, we get
\begin{align*}
\Big(\frac{G\times G}{U,\phi}\Big) \cong \tw_G^{\tilde\phi_1}\otimes_{AG} \mbox{\rm c}_{G,G}^\lambda\otimes\tw_G^{\tilde\phi_2}
\end{align*}
where $\tilde\phi = \tilde\phi_1\times\tilde\phi_2$ is an extension of $\phi$ to $G\times G$ and $\lambda\in \Aut(G)$.
It follows from the Mackey product formula that if $g\in G$ and $\mbox{\rm c}^g_{G,G}$ denotes the inner automorphism of $G$ induced by conjugation with 
$g$, then 
\begin{align*}
\mbox{\rm c}^g_{G,G}\otimes_{AG}\Big(\frac{G\times G}{U,\phi}\Big) \cong \Big(\frac{G\times G}{U,\phi}\Big)\otimes_{AG}\mbox{\rm c}^g_{G,G} \cong \Big(\frac{G\times G}{U,\phi}\Big).
\end{align*}
Hence, up to isomorphism, the automorphism $\lambda: G\to G$ can be taken as an outer automorphism. Furthermore, we have
\[
 \mbox{\rm c}_{G,G}^\lambda\otimes_{AG}\tw_G^{\phi}\cong \tw_G^{\phi\circ \lambda}\otimes_{AG} \mbox{\rm c}_{G,G}^\lambda
\]
for any $\phi\in G^*$ and $\lambda\in \out(G)$.
Thus the algebra $\bar E_G$ is generated by all $A$-fibered $(G,G)$-bisets of the form $\tw_G^{\phi}
\otimes_{AG} \mbox{\rm c}_{G,G}^\lambda$ where $\phi\in G^*$ and $\lambda\in \Out(G)$. For simplicity, we put
\[
[\phi,\lambda]_G := \tw_G^{\phi}\otimes_{AG} \mbox{\rm c}_{G,G}^\lambda.
\]
Now we define 
\[
\alpha : \bar E_G \to R[G^*\rtimes \out(G)]
\]
by $\alpha([\phi, \lambda]_G) = (\phi, \lambda)$. Clearly the linear extension of $\alpha$ is a well-defined isomorphism of abelian groups. We only show that it induces an algebra map. Given $[\phi, \lambda]_G, [\phi', \lambda']_G\in \bar E_G$, we have
\[
\alpha([\phi, \lambda]_G)\cdot\alpha([\phi', \lambda']_G) = (\phi, \lambda)\cdot (\phi', \lambda') = (\phi\cdot (\lambda\cdot \phi'), \lambda\lambda')
\]
by definition of the semidirect product. On the other hand, by the Mackey product formula, we have
\[
[\phi, \lambda]\cdot[\phi', \lambda'] = [\phi\cdot (\lambda\cdot \phi'), \lambda\lambda']_G,
\]
as required. \qed
\end{proof}

Now let $k$ be a field and $S$ be a simple fibered biset functor over $k$ with minimal group $G$. Clearly the evaluation $V = S(G)$ is a simple 
$\bar E_G$-module.  Hence, by the previous theorem, $V$ is a simple $k[G^*\rtimes \Out(G)]$-module. By the general theory explained in
\cite[Sections 2 and 4]{boucdecomposition}, given
a simple $\bar E_G$-module $V$, there is a unique simple fibered biset functor $S$ such that $S(G) = V$. We denote the simple functor corresponding 
to $V$ by $S_{G,V}$. Note that, in general, a simple fibered biset functor may have two non-isomorphic minimal groups. Thus at this point, we do not
know if the simple functor $S_{G,V}$ has another minimal group when $G$ is abelian. For the aims of this paper, this is not a problem. We remark that 
the exact situation can be determined using several results from \cite{boltjecoskun}.
\section{Fibered Burnside functor}

 In this section, we recall the ring structure on the fibered Burnside group from \cite{barker}, and introduce a natural fibered biset functor structure on it.
 We also determine the actions of basic fibered bisets on primitive idempotents of the fibered Burnside ring.
 \subsection{The ring structure}
 Let $G$ be a finite group and $A$ be an abelian group. The $A$-fibered Burnside group $B^A(G)$ can be identified with the free abelian group
 \[
 B^A(G) = \bigoplus_{[U,\phi]_G\in \mathcal M_G(A)/G}\mathbb Z \cdot[U,\phi]_G
 \] 
 on the set of $G$-conjugacy classes of pairs  $(U,\phi)\in \mathcal M_G(A)$.
Following \cite{dress}, we make $B^A(G)$ a ring via the linear extension of the following product. For $A$-fibered $G$-sets $X$ and $Y$, we define $X\cdot Y$ to be the union of $A$-orbits of $X\times Y$ with respect to the $A$-action
\[
a\cdot (x,y)=(a\cdot x,a^{-1}\cdot y)
\]
for $a\in A$ and $(x,y)\in X\times Y$. We denote the $A$-orbit containing the pair $(x, y)$ by $(x,_A y)$ and make the set $X\cdot Y$ an $A$-fibered $G$-set 
via $(g,a)\cdot(x,_A y) = (g\cdot a\cdot x,_A g\cdot y)$ for $(g,a)\in G\times A$ and $(x,_A y)\in X\cdot Y$. By \cite[Remark 2.3]{barker} and \cite[5.3]{boltje}, for $[U,\phi]_G$ and $[V,\psi]_G$ in $B^A(G)$, we have
\begin{equation}\label{eqn:productInB}
[U,\phi]_G\cdot [V,\psi]_G = \sum_{t\in [U\backslash G/V]} [U\cap {}^tV, \phi \cdot {}^t\psi]_G.
\end{equation}

The ring $B^A(G)$ is commutative and unital with unit $[G,1]_G$ and it extends the (ordinary) Burnside ring $B(G)$. Here we 
identify $B(G)$ with the subring of $B^A(G)$ generated by all the elements $[H,1]_G$ as $H$ runs over all subgroups of 
$G$. Moreover, following \cite[Remark 2.5.7]{bouc}, we can identify $B^A(G)$ with a subring of the $A$-fibered double Burnside ring 
$B(G\times G, A)$ as follows. Note that the ring structure on $B(G\times G, A)$ comes from the composition product of $A$-fibered $(G,G)$-bisets.

For an $A$-fibered $G$-set $X$, we define $A$-fibered $(G,G)$-biset $\Delta(X):=G\times X$ where the $A$-action is given by 
\[
a\cdot (g,x) = (g,a\cdot x)
\]
for any $a\in A$ and $(g,x)\in G\times X$ and the $(G,G)$-action is given by 
\[
g_1\cdot (g,x)\cdot g_2 = (g_1\cdot g\cdot g_2, g_2^{-1}\cdot x)
\]
for any $g,g_1,g_2\in G$ and $x\in X$. Note that for a transitive $A$-fibered $G$-set $[U,\phi]_G$, the above definition becomes
\[
\Delta([U,\phi])=\Big[ \frac{G\times G}{\Delta(U),\Delta(\phi)} \Big].
\]
Here, for any pair $(U,\phi)\in \mathcal M_G(A)$, the diagonal inclusion $(\Delta(U),\Delta(\phi))$ of $(U,\phi)$ in 
$\mathcal M_{G\times G}(A)$ is the pair consisting of the diagonal inclusion of the subgroup $U$ in $G\times G$ and
the diagonal homomorphism $\Delta(\phi)$ given by $\Delta(\phi)(u,u) = \phi(u)$. Extending this map linearly to the $A$-fibered Burnside group $B^A(G)$, we get a group homomorphism
\[
\Delta :  B^A(G) \longrightarrow B^A(G\times G).
\]
Note that the map $\Delta$ extends the 
map
\[
\Delta : B(G) \to B(G,G)
\]
defined in \cite[Lemma 2.5.8]{bouc}. Here $B(G)$ (resp. $B(G,G)$) denotes the Burnside ring (resp. the 
double Burnside ring) of the group $G$. Moreover we have the following result.

\begin{prop}
The map $\Delta$ is a unital ring homomorphism where $B^A(G\times G)$ is considered as a ring with the composition of fibered bisets.  
\end{prop}
\begin{proof}
It follows from the definition that the map $\Delta$ is a group homomorphism. We only prove that it is unital and it preserves products, that is, we prove that 
$\Delta$ is unital and for any $A$-fibered $G$-sets $X$ and $Y$, there is an isomorphism 
\[
\Delta(X\cdot Y) \cong \Delta(X)\otimes_{AG} \Delta(Y)
\]
of $A$-fibered $(G, G)$-bisets. Since $\Delta$ extends a similar map from the Burnside ring of $G$ to the double Burnside ring of $G$, it is unital. Next, we let
$X$ and $Y$ be $A$-fibered $G$-sets. Note that, by definition of $\Delta$, the $A\times G\times G$-set $\Delta(X)\times_{AG}\Delta(Y)$ is $A$-free and 
hence we have $\Delta(X)\times_{AG}\Delta(Y) = \Delta(X)\otimes_{AG}\Delta(Y)$ as $A$-fibered $(G,G)$-bisets. Now we define
\[
\beta: \Delta(X\cdot Y) \to \Delta(X)\otimes_{AG} \Delta(Y)
\]
by $\beta((g, (x,_A y))) = (g, x)\otimes (1, y)$. It is straightforward to show that $\beta$ is independent from the choice of the $A$-orbit representative $(x,y)$ of 
$(x,_A y)$. Moreover $\beta$ commutes with the $A\times G\times G$-action. Indeed, if $a\in A$ and $g_1, g_2\in G$, then
\begin{eqnarray*}
\beta((a,g_1)\cdot (g, (x,_A y))\cdot g_2) &=& \beta((g_1gg_2, a\cdot(x,_A y)\cdot g_2)) = \beta((g_1gg_2, (a\cdot g_2^{-1}\cdot x,_A g_2^{-1}\cdot y))) \\
&=& (g_1gg_2, a\cdot g_2^{-1}\cdot x)\otimes (1, g_2^{-1}\cdot y)\\ &=& ((g_1,a)\cdot(g, x)\cdot g_2)\otimes (1, g_2^{-1}\cdot y)
\\ &=& ((a, g_1)\cdot(g, x))\otimes (g_2, g_2^{-1}\cdot y)
\\ &=& ((a, g_1)\cdot(g, x))\otimes ((1, y)\cdot g_2)
\\ &=& (a, g_1)\cdot((g, x)\otimes (1, y))\cdot g_2
\\ &=& (a, g_1)\cdot\beta((g, (x,_{A} y)))\cdot g_2.
\end{eqnarray*}
Next we define the function
\[
\gamma: \Delta(X)\otimes_{AG} \Delta(Y) \to \Delta(X\cdot Y)
\]
given by $\gamma((g, x)\otimes (h, y)) = (gh, (h^{-1}\cdot x,_A y)).$ It is easy to check that $\gamma$ is an inverse to $\beta$ as a morphism of $A$-fibered
$(G,G)$-bisets. \qed
\end{proof}

We note that the above proof is very similar to proof of the first part of Lemma 2.5.8 in \cite{bouc}. With this result, we identify $B^A(G)$ with its image in 
$B(G\times G,A)$ via the map $\Delta$. Under this identification, given any $A$-fibered biset functor $F$, the evaluation $F(G)$ of $F$ at $G$
becomes a $B^A(G)$-module via restriction through $\Delta$. 

\subsection{The functorial structure}
Let $k$ be a field.  We denote by $kB^A$ the $A$-fibered biset functor over $k$ which associates a finite group $G$ to the $A$-fibered
 Burnside group $kB^A(G) = k\otimes B^A(G)$ and any $A$-fibered $(H,G)$-biset $X$ to the map
\[
kB^A(X): kB^A(G)\rightarrow kB^A(H).
\]  
given by left multiplication. To simplify the notation, we denote the map $kB^A(X)$ by $X$. Explicitly, if $(U,\phi)\in \mathcal M_G(A)$, then
\begin{equation}\label{eqn:actionOnB}
\Big[ \frac{H\times G}{V,\psi}\Big]\cdot [U,\phi]_G = \sum_{\substack{x\in [p_2(V)\backslash G/ U]\\ \psi_2\mid_{k_2(V)
\cap U} = {}^x\phi\mid_{k_2(V)\cap U}}} [V\ast {}^xU, \psi\ast {}^x\phi]_H.
\end{equation}

The functor $kB^A$ can be identified with the Yoneda functor Hom$_\mathcal C(- ,1)$ where $1$ denotes the trivial group. In particular, it is 
projective and by Yoneda's Lemma, we have End$_\mathcal F(kB^A) \cong kB^A(1) = k\otimes\mathbb Z \cong k$ as rings. In particular, the endomorphism 
ring of $kB^A$ is local, and hence the functor $kB^A$ is indecomposable. By the classification of simple $A$-fibered biset functors, it is the projective cover of 
the simple $A$-fibered biset functor $S_{1,k}$ where $k$ denotes the one dimensional $k$-vector space. The following proposition is a crucial result that we need to describe the other composition factors of $kB^A$.
\begin{prop}\label{ideal}
The evaluation $kB^A(G)$ of $kB^A$ at $G$ regarded as a module over $kB^A(G)$ via restriction through $\Delta$ is the regular $kB^A(G)$-module.
In particular, for any subfunctor $F$ of $kB^A$ and for any finite group $G$, the evaluation $F(G)$ is an ideal of $kB^A(G)$. 
\end{prop}
\begin{proof}
To prove the first assertion, it suffices to prove that there is an isomorphism of $A$-fibered $G$-sets
$$\Delta(X)\otimes_{AG} Y \cong X\cdot Y$$ 
for any $A$-fibered $G$-sets $X,Y\in B^A(G)$. Following the proof of Lemma 2.5.8 of \cite{bouc}, we define
\[
\alpha: \Delta(X)\otimes_{AG} Y \to X\cdot Y
\]
by $\alpha((g,x)\otimes y) = (g\cdot x,_A g\cdot y)$. This is well-defined with inverse 
\[
\beta: X\cdot Y\to \Delta(X)\otimes_{AG} Y
\]
given by $\beta((x,_A y)) = (1,x)\otimes y$. We leave the details to reader and note that the proofs are almost identical to those in \cite{bouc}, one only needs
to check the $A$-action. Now the second part of the statement follows from the first part since the evaluation $F(G)$ becomes a submodule of the regular 
$kB^A(G)$-module. \qed
\end{proof}
\begin{rem}
The second part of the above result holds for an arbitrary Green biset functor, that is, given a Green biset functor $A$, as in Section 8.5 of \cite{bouc}, an 
$A$-submodule $F$ of $A$ and a finite group $G$, then $F(G)$ is an ideal of $A(G)$ with the ring structure on $A(G)$ induced from the Green biset functor 
structure.
\end{rem}

\subsection{Actions on idempotents}

With Proposition \ref{ideal}, we see that there is a connection between subfunctors of the fibered Burnside functor $kB^{A}$ and the 
structure of the Burnside ring $kB^A(G)$. To make this relation explicit, we need to determine the action of basic $A$-fibered 
biset operations on primitive idempotents of the ring $kB^A(G)$. There are two formulas for primitive idempotents of $kB^A(G)$, see \cite{barker}
and \cite{boltje}. In this paper, we use Barker's formula.

Let $A$ be a cyclic $p$-group and suppose $k$ is a sufficiently large field of characteristic $0$. Also let $O(G) = O_A(G)$ 
denote the intersection of kernels of all homomorphisms $G\rightarrow A$. In this case, the group $G^*$ is isomorphic with the dual group Hom$(G/O(G), A)$ of the 
group $G/O(G)$. We denote by $\mathcal E_G(A)$ the set of all pairs $(H,h)$ where $H\leq G$ and $h$ runs over a 
complete set of left coset representatives of $O(H)$ in $H$. The set $\mathcal E_G(A)$ is a $G$-set via conjugation, and we 
have $|\mathcal{M}_{G}(A)|=|\mathcal E_G(A)|$ by \cite[Lemma 3.1]{barker}. It turns out that there is a bijective correspondence between the set of primitive idempotents
of $kB^A(G)$ and the set $\mathcal E_G(A)$. Writing $e_{H,h}^G$ for the primitive idempotent corresponding to the pair $(H,h)$, 
\cite[Theorem 5.2]{barker} gives
\begin{equation}\label{barkerFormula}
e^{G}_{H,h} = \frac{1}{\lvert N_{G}(H,h) \rvert}\sum_{(V, \nu) \in \mathcal{M}_{G}(A)/G} \lvert V \rvert 
\mu _{G}(V,\nu; H,h)[V, \nu]_{G}
\end{equation}  
where $N_G(H,h)$ denotes the stabilizer in $G$ of the pair $(H,h)$ under the above action of $G$ and where
\begin{align*}
\mu_{G}(V,\nu;H,h)=\sum_{(V',\nu^\prime)\in [V,\nu]_{G}} \nu'^{ -1}(V^\prime\cap hO(H))\mu(V^\prime,H)/\left|V^\prime
\right|
\end{align*}
is the monomial Möbius function and the above sum is over all pairs $(V', \nu')$ $G$-conjugate to $(V, \nu)$.

The idempotents $e_{H, h}^G$ can also be characterized using the algebra maps $kB^A(G)\to k$, called \textit{species}, introduced by Dress \cite{dress}. The 
set of all species is known to be a basis for the dual space of the fibered Burnside ring, see \cite{dress} or \cite[Lemma 5.1]{barker}.
For any $(H,h)\in \mathcal E_G(A)$, we denote the species associated to the $G$-conjugacy class of $(H,h)$ by
\begin{align*}
s^{G}_{H,h}: kB^{A}(G)\rightarrow k
\end{align*}
which is given, for an $A$-fibered $G$-set $X$, by
\begin{align*}
s^{G}_{H,h}[X]= \sum_{Ax}\phi_{x}(h).
\end{align*}
Here the index $Ax$ runs over the fibers of $X$ that are stabilized by $H$ and given such an orbit $Ax$, the $H$-action on $Ax$ induces a group 
homomorphism $\phi_x: H\to A$ given by $\phi_x(h) = a$ provided that $h\cdot x = a\cdot x$ for some (unique) $a\in A$. Now by the duality, the idempotent $e^{G}_{H,h} \in kB^A(G)$ is the 
unique element such that
\begin{align*}
s^{G}_{K,k}(e^{G}_{H,h})= \begin{cases} 
      1 &  (H,h)=_{G}(K,k)\\
      0 & \mbox{\rm otherwise.}
   \end{cases}
 \end{align*}
In particular, $e^{G}_{H,h}$ is the only non-zero idempotent of $kB^A(G)$ such that $X\cdot e^{G}_{H,h}=s^{G}_{H,h}(X)e^{G}_{H,h}$ for any $X\in 
kB^A(G)$. Thus, for any $X\in kB^A(G)$, we have the coordinate decomposition
\begin{align*}
X= \sum_{(H,h) \in \mathcal E_G(A)/G}s^{G}_{H,h}(X)e^{G}_{H,h}.
\end{align*}

\begin{rem}
Let $G$ be a $p$-group. The above idempotent formula still holds if we replace $k$ with a sufficiently large field of characteristic $q\neq p$ since all the denominators of the formula are $p$-powers and hence invertible in $k$. Thus when we restrict to the fibered $p$-biset
functor of fibered Burnside ring over such fields, the above formula will remain valid.
\end{rem}

Our next aim is to determine the actions of basic fibered bisets on primitive idempotents. Two of them are described by Barker in \cite[Proposition 5.4 and Proposition 5.5]{barker}. We recall
the result.
\begin{prop}[Barker]
Let $K\le G$ be finite groups. Then
\begin{enumerate}
\item[(i)] For $(H,h)\in \mathcal E_G(A)$, we have
$$\Res^{G}_{K} e^{G}_{H,h}=\sum_{(J,j)} e^{K}_{J,j}$$
where the pairs $(J,j)$ run over representatives of the $K$-classes of the pairs which are $G$-conjugate to $(H,h)$.
\item[(ii)] Let $(H,h)\in\mathcal E_K(A)$. Then 
\begin{align*}
\Ind^{G}_{K} e^{K}_{H,h}=\lvert N_{G}(H,h):N_{K}(H,h)\rvert e^{G}_{H,h}.
\end{align*}
\end{enumerate} 
\end{prop}
 
 The following Frobenius type formula is used in the next theorem to describe actions of inflation and deflation. 

\begin{lemma} \mbox{\rm{(Frobenius Relation)}}
Let $N \unlhd G$ be a normal subgroup of $G$. Then, for any $A$-fibered $G$-set $X$ and any $A$-fibered $G/N$-set $Y$, we have an isomorphism of 
$A$-fibered $G/N$-sets
\begin{align*}
Y\cdot \Def^{G}_{G/N}(X)\cong\Def^{G}_{G/N}\left(\Inf^{G}_{G/N}(Y)\cdot X\right).
\end{align*}
\end{lemma}
\begin{proof}
Recall that the deflation map $\Def^{G}_{G/N}$ (resp. the inflation map $\Inf^{G}_{G/N}$) corresponds to composition with the $(G/N, G)$-biset $G/N$
(resp. $(G, G/N)$-biset $G/N$). Thus, given an $A$-fibered $G$-set $X$ and an $A$-fibered $G/N$-set $Y$, we have to prove
\[
Y\cdot (G/N \otimes_{AG} X)\cong G/N\otimes_{AG} \left((G/N\otimes_{A(G/N)} Y)\cdot X\right).
\]
Moreover by the proof of Proposition \ref{ideal}, we can rewrite the above equality as 
\[
\Delta(Y)\otimes_{A(G/N)} (G/N \otimes_{AG} X)\cong G/N\otimes_{AG} \left(\Delta(G/N\otimes_{A(G/N)} Y)\otimes_{AG} X\right).
\]
Hence by the associativity of the product $\otimes_{AG}$, it is sufficient to prove that
\[
\Delta(Y)\otimes_{A(G/N)} G/N\cong G/N\otimes_{AG} \Delta(G/N\otimes_{A(G/N)}Y)
\]
as $A$-fibered $(G/N,G)$-sets. This is a straightforward generalization of part 1 of Proposition 2.5.10 in \cite{bouc}. We leave the details to reader. \qed
\end{proof}

\begin{theorem}\label{deflationconstant}
Let $N \unlhd G$ be a normal subgroup of $G$. 
\begin{enumerate}
\item[(i)] For any $(H/N,hN) \in \mathcal E_G(A)$, we have
\begin{align*}
\Inf^{G}_{G/N}(e^{G/N}_{H/N,hN})=\sum_{(K,k)} e^{G}_{K,k}
\end{align*}
where $(K,k)$ runs over representatives of the $G$-classes of $\mathcal E_G(A)$ such that $(KN,k)$ is $G$-conjugate to $(H,h)$.
\item[(ii)] For any $(H,h)\in \mathcal E_G(A)$, we have
\begin{align*}
\Def^{G}_{G/N}e^{G}_{H,h} = m_{H,h}^G\cdot e^{G/N}_{HN/N,hN} 
\end{align*}
for some constant $m_{H,h}^G$.
\end{enumerate}
\end{theorem}
\begin{proof} 
First we demonstrate that for any $(K,k)\in \mathcal E_G(A)$ and for any $A$-fibered $G/N$-set $S$, we have 
$s^{G}_{K,k}(\Inf^{G}_{G/N}(S))=s^{G/N}_{KN/N,kN}(S)$. Since the inflation map is a group homomorphism, it suffices to take $S$ 
transitive. For any transitive $A$-fibered $G/N$-set $[V/N, \nu]_{G/N}$, Equation (3) implies
\[
\Inf^{G}_{G/N}([V/N, \nu]_{G/N}) = G/N\otimes_{A(G/N)} [V/N, \nu]_{G/N} = [V, \bar\nu]_G
\]
where $\bar\nu$ denotes the inflation of $\nu$ to $V$. Thus, we have
\begin{align*}
s^{G}_{K,k}(\Inf^{G}_{G/N}([V/N, \nu]_{G/N}))=s^{G}_{K,k}([V, \bar{\nu}]_{G}) = \sum_{Ag} \bar{\nu}_{g}(k)
\end{align*}
where $Ag$ runs over the fibers stabilized by $K$. But these fibers are also stabilized by $KN/N$ and we have 
$\bar{\nu}_{g}(k)=\nu_{g}(kN)$. Hence
\begin{align*}
s^{G}_{K,k}(\Inf^{G}_{G/N}([V/N, \nu]_{G/N}))=s^{G/N}_{KN/N,kN}([V/N, \nu]_{G/N}).
\end{align*}
Therefore, we obtain
\begin{align*}
s^{G}_{K,k}(\Inf^{G}_{G/N}(e^{G/N}_{H/N, hN}))&= s^{G/N}_{KN/N,kN}(e^{G/N}_{H/N, hN})=  \left\{
\begin{array}{ll}
      1 &  ~\text{if}~ (KN/N, kN)=_{G/N} (H/N,hN),\\
      0 &  ~\text{otherwise}\vspace{0.1in} \\
\end{array} 
\right.
\\&= \left\{
\begin{array}{ll}
      1 &  ~\text{if}~  (KN, k)=_{G} (H,h),\\
      0 &  ~\text{otherwise}\vspace{0.1in} \\
\end{array} 
\right. 
\end{align*}
and the first part follows.

For the second part, let $S$ be an arbitrary $A$-fibered $G/N$-set. Then, using the Frobenius relation, we obtain
\begin{align*}
S\cdot \Def^{G}_{G/N} e^{G}_{H,h} & = \Def^{G}_{G/N}(\Inf^{G}_{G/N}(S)\cdot e^{G}_{H,h})  = 
\Def^{G}_{G/N}(s^{G}_{H,h}(\Inf^{G}_{G/N}S)\cdot e^{G}_{H,h}) \\ & = \Def^{G}_{G/N}(s^{G/N}_{HN/N,hN}(S)\cdot e^{G}_{H,h})  
= s^{G/N}_{HN/N,hN}(S)\cdot \Def^{G}_{G/N} (e^{G}_{H,h}).
\end{align*} 
However, $e^{G/N}_{HN/N,hN}$ is the unique element with the above property. Therefore, we conclude that
\begin{align*}
\Def^{G}_{G/N}e^{G}_{H,h} = m\cdot e^{G/N}_{HN/N,hN} 
\end{align*}
for some constant $m$. \qed
\end{proof}

Finally we describe the actions of transport of structure and twist bisets. We skip the straightforward proofs.
\begin{prop}\label{pro:iso-tw}
Let $\lambda:G\rightarrow G^\prime$ be a group isomorphism, let $\phi\in G^*$ and let $(H,h)\in \mathcal E_G(A)$. Then
\begin{enumerate}
\item[\mbox{\rm (i)}] $c_{G,G^\prime}^\lambda e_{H,h}^G = e_{\lambda(H),\lambda(h)}^{G^\prime}$,
\item[\mbox{\rm (ii)}] $\tw_{G}^\phi e_{H,h}^G = \phi(h)\cdot e_{H,h}^{G}$.
\end{enumerate} 
\end{prop}

\section{Composition factors over $p$-groups}\label{composition}

In this section, we restrict our attention to the category of $A$-fibered $p$-biset functors and determine the 
subfunctors of the fibered Burnside functor $kB^A$ over a field $k$ of characteristic $q\neq p$. Here by an 
$A$-fibered $p$-biset functor we mean a $k$-linear functor $\mathcal C_p \rightarrow {}_k$Mod where 
$\mathcal C_p$ 
is the full subcategory of $\mathcal C$ consisting only of $p$-groups, for a fixed prime $p$.
We also assume that $A$ is a subgroup of the unit group of a field of characteristic zero and that the order of $A$ is divisible by 
$p$. 
Our approach is similar to that of Bouc and Th\'evenaz \cite[Section 8]{boucthevenaz}. 

The precise situation is as follows. We fix a prime $p$ and a positive integer $n$. We denote by 
$\mu_n$ a cyclic group of order $p^n$. For a $p$-group $P$, we let $O_n(P)$ denote the group $O_{\mu_n}(P)$. 
Finally we let $k$ be an algebraically closed field of characteristic $q\neq p$ and fix an embedding $\mu_n\to k^\times$.
 Our aim is to determine subfunctors of the functor $k\otimes B^{\mu_n}$. For simplicity write $kB^n:=k\otimes B^{\mu_n}$. For 
 this aim, we first determine minimal groups for the subfunctors of the fibered Burnside functor. 
 
Interestingly, in all cases, the minimal groups are elementary abelian. As in the case of the ordinary Burnside functor, extra work 
should be done to see which elementary abelian groups appear as a minimal group and it turns out that the possible ranks depends only on whether $q| p-1$.

To begin with, let $F$ be a subfunctor of $kB^n$ and suppose that $G$ is a minimal group of $F$. We know, by Proposition \ref{ideal}, that $F(G)$ is an 
ideal of 
$kB^n(G)$. Therefore, it is generated by a set of primitive idempotents $e^{G}_{H,h}$ of $kB^n(G)$. Let $X$ be a fibered $(L,G)$-biset for some group 
$L$. If $X$ can be factored through a group $K$ with $\left|K\right|<\left|G\right|$, then for any $e^{G}_{H,h}\in F(G)$ we should have $X\cdot 
e^{G}_{H,h}=0$. This implies that to find the minimal groups, we need a deeper understanding of the action on the idempotents of the fibered bisets 
that map to groups of smaller order.    

First, notice that for any proper subgroup $H < G$, the idempotent $e^{G}_{H,h}$ is not contained in $F(G)$. Indeed
$\Res^{G}_{H}e^{G}_{H,h}=e^{H}_{H,h}$ and $F(H) =0$. Therefore, the ideal $F(G)$ must be generated by idempotents of the form $e_{G,g}^G$. 

Next we consider the action of deflation maps on the idempotents of the form 
$e^{G}_{G,g}$. Recall that if $N \unlhd G$ is a normal subgroup of $G$, then 
\begin{equation}\label{defconstant}
\Def^{G}_{G/N}e^{G}_{G,g} = m\cdot e^{G/N}_{G/N,gN}
\end{equation}
for some constant $m$. Since $G$ is minimal, for any non-trivial normal subgroup $N$ of $G$ the constant $m$ should be zero. 
Following Bouc and Thev\'{e}naz \cite{boucthevenaz}, we consider the elementary abelian $p$-groups and non-elementary abelian $p$-groups, separately. Let $\Phi(G)$ denote the Frattini subgroup of $G$ and $\overline G$ denote the quotient $G/\Phi(G)$. 

\begin{lemma}\label{lem:def-p}
For any $p$-group $G$ and $g\in G$, we have
\begin{align*}
\Def^{G}_{\overline G}e^{G}_{G,g} = \frac{\lvert O(G) \rvert}{\lvert N_{G}(G,g) \rvert}\cdot \lvert \overline G\rvert\cdot 
e^{\overline G}_{\overline G,g\Phi (G)}.
\end{align*}
\end{lemma}

\begin{proof}
Recall the idempotent formula 
\begin{align*}
e^{G}_{G,g} = \frac{1}{\lvert N_{G}(G,g) \rvert}\sum_{(V, \nu) \in_{G} \mathcal{M}_{G}(A)} \lvert V \rvert \mu _{G}(V,\nu; G,g)[V, \nu]_{G}.
\end{align*}
By the definition of the deflation maps, we have
\begin{align*}
\Def^{G}_{\overline G}e^{G}_{G,g} = \frac{1}{\lvert N_{G}(G,g) \rvert}\sum_{\substack{(V, \nu) \in_{G} \mathcal{M}_{G}(A)\\ V\cap \Phi (G)\leq \mbox{\rm ker}\nu}} \lvert V \rvert \mu _{G}(V,\nu; G,g)[V\Phi (G)/\Phi (G), \bar{\nu}]_{\overline G}. 
\end{align*}
Then Equation (\ref{defconstant}) becomes
\begin{align*}
&\frac{1}{\lvert N_{G}(G,g) \rvert}\sum_{\substack{(V, \nu) \in_{G} \mathcal{M}_{G}(A)\\ V\cap \Phi (G)\leq ker\nu}} \lvert V \rvert 
\mu _{G}(V,\nu; G,g)[V\Phi (G)/\Phi (G), \bar{\nu}]_{\overline G} \\&
= \frac{m}{\lvert N_{\overline G}(\overline G,g\Phi (G)) \rvert}\sum_{(W, \omega)} \lvert W \rvert 
\mu _{\overline G}(W,\omega; \overline G,g\Phi (G))[W, \omega]_{\overline G}.
\end{align*}
Here $(W,\omega)$ runs over a set of representatives of the $\overline G$-conjugacy classes of the set 
$\mathcal M_{\overline G}(A)$. Now the coefficient of $[\overline G, 1]_{\overline G}$ in the right-hand side is
\begin{align*}
 \frac{m}{\lvert N_{\overline G}(\overline G,g\Phi (G)) \rvert} \lvert \overline G \rvert 
 \mu _{\overline G}(\overline G,1; \overline G,g\Phi (G)).
\end{align*}
We also have
\begin{align*}
\mu _{\overline G}\big(\overline G,1; \overline G,g\Phi (G)\big) = \sum_{(V,\nu) \in [\overline G, 1]_{\overline G}} \big\lvert V\cap 
g\Phi(G)O\big(\overline G\big) \big\rvert \mu (V,\overline G)/\lvert V \rvert = \frac{1}{\lvert \overline G\rvert}. 
\end{align*}
Here the last equality holds since the only pair which is $\overline G$-conjugate to $[\overline G, 1]_{\overline G}$ is $(\overline G, 1)$ and hence the sum collapses to the term
$ \big\lvert \overline G\cap g\Phi(G)O\big(\overline G\big) \big\rvert \mu (\overline G,\overline G)/\lvert \overline G \rvert$ and the intersection 
$\overline G\cap g\Phi(G)O\big(\overline G\big)$ consist only of the element $g\Phi(G)$.  
Thus, the coefficient is 
\begin{align*}
\frac{m}{\lvert N_{\overline G}(\overline G,g\Phi (G)) \rvert}=\frac{m}{\lvert \overline G\rvert}.
\end{align*}
On the other hand, the coefficient of $[\overline G, 1]_{\overline G}$ on the left-hand side is
\begin{align*}
\frac{1}{\lvert N_{G}(G,g) \rvert}\sum_{V\Phi (G)=G} \lvert V \rvert \mu _{G}(V,1; G,g).
\end{align*}
Since $\Phi (G)$ is the Frattini subgroup of $G$, the equality $V\Phi (G)=G$ implies that $V=G$. Then 
\begin{align*}
\mu _{G}(G,1; G,g) = \sum_{(W,1) \in [G, 1]_{G}} \lvert W\cap gO(G)) \rvert \mu (W,G)/\lvert W \rvert= \lvert O(G) \rvert/\lvert
 G\rvert.
\end{align*}
Therefore we get
\begin{align*}
m=\frac{\lvert O(G) \rvert}{\lvert N_{G}(G,g) \rvert}\cdot \lvert G/\Phi(G)\rvert
\end{align*}
as required. \qed
\end{proof}

Now let $F$ be a subfunctor of $kB^n$ and $G$ be a minimal group for $F$.
If $G$ is not elementary abelian, then the Frattini subgroup $\Phi(G)$ of $G$ is nontrivial. Also, the coefficient 
$m=\frac{\lvert O(G) \rvert}{\lvert N_{G}(G,g) \rvert}\cdot \lvert \overline G\rvert$ in the previous lemma is non-zero. Indeed, all the 
terms $\lvert O(G) \rvert$, $\lvert \overline G \rvert$ and $\lvert N_{G}(G,g) \rvert$ are orders of subgroups of the $p$-group $G$ and hence 
$m$ is a power of $p$ and the characteristic $q$ of the field $k$ is not equal to $p$. But since the
deflation of $e_{G,g}^G$ to a non-trivial normal subgroup is non-zero, we conclude that $G$ is not a minimal group. This proves the 
following result.

\begin{prop}\label{non-elt} 
Let $F$ be a subfunctor of $kB^n$ and $G$ be a minimal group for $F$. Then $G$ is elementary abelian.
\end{prop}

Our next goal is to find which elementary abelian $p$-groups can be a minimal group of a subfunctor of $kB^{n}$. We first evaluate the deflation map
on the primitive idempotents $e_{G,g}^G$ when $G$ is an elementary abelian $p$-group.

\begin{prop}\label{minimalconstant}
Let $G$ be an elementary abelian $p$-group of rank $r$, let $h$ be a non-trivial element of $G$ and $H=<h>$ be the subgroup generated by $h$. Then, we have
\begin{align}
\Def^{G}_{G/H}e^{G}_{G,g}= m_{g,H}\cdot e^{G/H}_{G/H,gH}
\end{align}
where 
$$
m_{g,H} =   \left\{
\begin{array}{ll}
      \frac{1-p^{r-1}}{p} &  ~\text{if}~ g=1, \vspace{0.1in}\\
      \frac{1}{p} &  ~\text{if}~ g\neq 1, g\in H,\vspace{0.1in} \\
      \frac{1-p^{r-2}}{p} &  ~\text{if}~ g\notin H. \\
\end{array} 
\right.
$$
\end{prop}

\begin{proof}
Recall the idempotent formula
\begin{align*}
e^{G}_{G,g}=\frac{1}{\lvert N_{G}(G,g) \rvert}\sum_{(V, \nu) \in_{G} \mathcal{M}_{G}(A)} \lvert V \rvert \mu _{G}(V,\nu; G,g)[V,\nu]_{G} .
\end{align*}
Since $G$ is elementary abelian, and hence abelian, the $G$-actions seen in the formula are all trivial and hence the formula
becomes
\begin{align*}
e^{G}_{G,g}=\frac{1}{\lvert G \rvert}\sum_{(V, \nu) \in \mathcal{M}_{G}(A)} \lvert V \rvert \mu(V,\nu; G,g)[V,\nu]_{G} .
\end{align*}
Furthermore, since $G$ is a non-trivial elementary abelian $p$-group, the subgroup $O(G)$ is the same as the intersection of kernels of the
irreducible complex characters, and hence it is trivial. Therefore,
\[
\mu(V,\nu; G,g) =\nu(g) \mu(V,G)/|V|
\]
and we have 
\begin{align*}
e^{G}_{G,g}=\frac{1}{\lvert G \rvert}\sum_{(V, \nu) \in \mathcal{M}_{G}(A)} \nu(g) \mu(V,G)[V,\nu]_{G} .
\end{align*}
Now we need to apply the deflation map to both sides. As in the proof of Lemma \ref{lem:def-p}, we have
\begin{align*}
\Def^{G}_{G/H}e^{G}_{G,g} = \frac{1}{\lvert G \rvert}\sum_{\substack{(V, \nu) \in \mathcal{M}_{G}(A)\\ V\cap H\leq \mbox{\rm ker}\nu}} \nu(g)\mu(V, G)[VH/H, \bar{\nu}]_{G/H}. 
\end{align*}
To evaluate the right hand side of the above equality, we consider three separate cases, namely, $g=1$ or $1\neq g\in H$ or $g\notin H$.
For the first case, suppose $g=1$. Then the above equality becomes
\begin{align*}
\Def^{G}_{G/H}e^{G}_{G,1} = \frac{1}{\lvert G \rvert}\sum_{\substack{(V, \nu) \in \mathcal{M}_{G}(A)\\ V\cap H\leq \mbox{\rm ker}\nu}} \mu(V, G)[VH/H, \bar{\nu}]_{G/H}. 
\end{align*}
On the other hand, by Equation \ref{defconstant}, the left hand side of the above equation is also equal to a multiple of
\begin{align*}
e^{G/H}_{G/H,H}=\frac{1}{\lvert G:H \rvert}\sum_{(V, \nu) \in \mathcal{M}_{G/H}(A)}  \mu(V,G/H)[V,\nu]_{G/H}. 
\end{align*}
Thus, we have the equality
\begin{align*}
\frac{1}{\lvert G \rvert}\sum_{\substack{(V, \nu) \in \mathcal{M}_{G}(A)\\ V\cap H\leq \mbox{\rm ker}\nu}}  \mu(V,G)[VH/H,\bar{\nu}]_{G/H}=\frac{m}{\lvert G:H \rvert}\sum_{(W, \omega) \in \mathcal{M}_{G/H}(A)}  \mu(W,G/H)[W,\omega]_{G/H}.
\end{align*} 
To determine the constant $m$, we compare the coefficients. Note that the coefficient of $[G/H, 1]_{G/H}$ in the right hand side 
is $\frac{m}{p^{r-1}}$. In the left hand side, it is 
$$\frac{1}{\lvert G \rvert} \sum_{V}  \mu(V,G)$$ 
where $V$ runs over the subgroups satisfying $VH = G$. But, the last equality implies either that $V=G$ or that $V$ is a complement of $H$ in 
$G$. If $V$ is a complement of $H$, then $\lvert V \rvert=p^{r-1}$. But, in this case, $V$ is maximal subgroup of $G$. Thus, 
$\mu(V,G)=-1$. Note that there are $p^{r-1}$ many complements of $H$. If $V=G$, then obviously we have $\mu(V,G)=1$. 
Therefore, the coefficient in the left hand side becomes $\frac{1-p^{r-1}}{p^{r}}$. We conclude that $m=\frac{1-p^{r-1}}{p}$, as 
required.

For the second case, we let $1\neq g\in H$. As above, we have 
\begin{eqnarray*}
\frac{1}{\lvert G \rvert}\sum_{\substack{(V, \nu) \in \mathcal{M}_{G}(A)\\ V\cap H\leq \mbox{\rm ker}\nu}} \nu^{-1}(g)  \mu(V,G)[VH/H,\bar{\nu}]_{G/H}&=&\\ \frac{m}{\lvert G:H \rvert}\sum_{(W, \omega)}\omega^{-1}(gH)  \mu(W,G/H)[W,\omega]_{G/H}. 
\end{eqnarray*}

The coefficient of $[G/H, 1]_{G/H}$ in the right hand side is again $\frac{m}{p^{r-1}}$. In the left hand side, it is equal to the sum 
$$\frac{1}{\lvert G \rvert} \sum_{V}  \mu(V,G)$$
where $V$ runs over subgroups containing $g$ and satisfying $VH=G$. Note that since $g \in H\cap V$ and $g \neq 1$, the 
subgroup  $V$ cannot be a complement of $H$ and hence we must have $V=G$. Therefore the coefficient becomes 
$\frac{1}{p^{r}}$ and we hence we get that $m=\frac{1}{p}$, as required.

The last case where $g\notin H$ is similar to the above cases. We do not include the proof.\qed 
\end{proof}

With this result, we can determine the subfunctors and minimal groups more explicitly. 
Let $\mathcal{I}=\{0\}\cup \{r\in \mathbb{N}\mid p^{r-1}\equiv 1~(\text{mod}~q)\}$ be the set of powers $r$ of $p$ for which all 
proper deflations by cyclic subgroups of the idempotent $e_{E,1}^E$, with $E$ elementary abelian of rank $r$, is zero. We enumerate the elements of 
$\mathcal{I}=\{r_{i}\}_{i=0}^{\infty}$ such that $i<j$ implies $r_{i}<r_{j}$. Then we have the following theorem.

\begin{theorem}\label{generator}
Let $F$ be a subfunctor of $kB^{n}$ and $G$ be a minimal group of $F$. Then, 
\begin{enumerate}
\item[\mbox{\rm (i)}] the group $G$ is elementary abelian of order $p^{r}$, for some $r\in \mathcal I$,
\item[\mbox{\rm (ii)}] the $k$-vector space $F(G)$ is 1-dimensional generated by $e^{G}_{G,1}$ and
\item[\mbox{\rm (iii)}] the subfunctor $F$ is generated by $e_{G,1}^G$.
\end{enumerate}
\end{theorem}

\begin{proof}
By Proposition \ref{non-elt}, we know that $G$ is elementary abelian. We also know that $F(G)$ is generated by idempotents of 
the form $e^{G}_{G,g}$. Suppose that the idempotent $e^{G}_{G,g}$ is contained in $F(G)$ for some $g\neq 1$. Then by Proposition
 \ref{minimalconstant}, we have
$$0 \neq \Def^{G}_{G/<g>}e^{G}_{G,g}=\frac{1}{p}\cdot e^{G/<g>}_{G/<g>,<g>}\in F(G/<g>).$$
 But $G$ is minimal and hence $F(G/<g>) = 0$, contradiction. Therefore, the idempotent $e^{G}_{G,1}$ must generate $F(G)$. Moreover, if $G$ has rank $r$, then, by Proposition \ref{minimalconstant}, 
$$\Def^{G}_{G/<g>}e^{G}_{G,1}=\frac{1-p^{r-1}}{p}\cdot e^{G/<g>}_{G/<g>,<g>}\in F(G/<g>).$$ 
Thus, we must have $p^{r-1}\equiv 1~\text{mod}~q$. Thus we have proved the first two parts of the theorem. 

To prove the last part, let $G = E_r$ be an elementary abelian group of rank $r$ and  
suppose, for a contradiction, that the idempotent $e^{E_{r}}_{E_{r},1}$ does not generate $F$ and let $K$ denote the subfunctor of 
$F$ generated by $e^{E_{r}}_{E_{r},1}$. Then, there exists a group $T$ such that for some element $x\in F(T)$, we have 
$x\notin K(T)$. Suppose $T$ has minimal order with respect to the property that $F(T)\neq K(T)$. Since $F(T)$ is an ideal of $kB^{n}(T)$, it is generated by a set $I$ of primitive idempotents. Thus, in the primitive 
idempotent basis, we have 
$$x=\sum_{I}x^{T}_{H,h}\cdot e^{T}_{H,h}$$
for some constants $x^{T}_{H,h}\in k$. This implies that for some pair $(H,h)$, the idempotent $e^{T}_{H,h}$ is not contained in 
$K(T)$.  

To determine the set $I$, suppose that $e^{T}_{H,h}\notin K(T)$ and $H\neq T$. Then, by the minimality of $T$, we have 
$\Res^{T}_{H} e^{T}_{H,h} \in K(H)$. So, for some $X\in kB^n(H\times E_r)$, we have 
$$\Res^{T}_{H} e^{T}_{H,h}= X\cdot e^{E_{r}}_{E_{r},1}.$$
Thus, multiplying both sides by an induction biset, we get 
$$\Ind^{T}_{H}\Res^{T}_{H} e^{T}_{H,h}=(\Ind^{T}_{H}X)e^{E_{r}}_{E_{r},1}.$$ 
But note that 
$$e^{T}_{H,h}\cdot \Ind^{T}_{H}\Res^{T}_{H} e^{T}_{H,h}=\alpha\cdot e^{T}_{H,h}$$ 
for some non-zero $\alpha\in k$. Thus, we have 
$$e^{T}_{H,h}\cdot \big((\Ind^{T}_{H}X)e^{E_{r}}_{E_{r},1}\big)=\alpha\cdot e^{T}_{H,h}$$ 
which implies that $e^{T}_{H,h}=\frac{1}{\alpha}\cdot (\widetilde{e^{T}_{H,h}}\cdot \Ind^{T}_{H}X)\cdot e^{E_{r}}_{E_{r},1}$. 
This is a contradiction since we assumed that $e_{H,h}^T$ is not in $K(T)$. So we must have $H=T$.

Next suppose that $e^{T}_{T,t}\notin K(T)$ and the Frattini subgroup $\Phi(T)$ is non-trivial. Then, again, by the minimality of $T$,
we have $\Def^{T}_{T/\Phi(T)}e^{T}_{T,t}=X\cdot e^{E_{r}}_{E_{r},1}$ for some $X\in kB^n(T/\Phi(T)\times T)$. Note that 
$$e^{T}_{T,t}\cdot (\Inf^{T}_{T/\Phi(T)}\Def^{T}_{T/\Phi(T)}e^{T}_{T,t})=\beta\cdot e^{T}_{T,t}$$ for some non-zero $\beta\in k$. 
Thus, we have $$e^{T}_{T,t}\cdot (\Inf^{T}_{T/\Phi(T)}X\cdot e^{E_{r}}_{E_{r},1})=\beta\cdot e^{T}_{T,t}$$ which implies 
$e^{T}_{T,t}=\frac{1}{\beta}\cdot (\widetilde{e^{T}_{T,t}}\cdot \Inf^{T}_{T/\Phi(T)}X)\cdot e^{E_{r}}_{E_{r},1}$ which is again a 
contradiction. So, we must have $\Phi(T)=1$ and $T$ is elementary abelian. 

Finally, suppose $e^{T}_{T,t}\notin K(T)$ and $T$ is elementary abelian. By part (i), the rank of $T$ is greater than or equal to that of $G$ and
hence there is a subgroup $U$ of $T$ isomorphic to $G$. Without loss of generality, suppose $U= G$. Now if $t\neq 1$, then we 
have
$$\Def^{T}_{T/<t>}e^{T}_{T,t}=\frac{1}{p}\cdot e^{T/<t>}_{T/<t>,<t>}$$ 
and by similar arguments, we again obtain a contradiction. Therefore, we must have $t=1$. Also, if $P\le R$ are elementary 
abelian, then the idempotent $e_{R,1}^R$ is a summand of $\Inf_{P}^R e_{P,1}^P$. In particular, the idempotent $e^{T}_{T,1}$ is a 
summand of $\Inf^{T}_{E_{r}}e^{E_{r}}_{E_{r},1}$ and hence in $K(T)$, a contradiction. This 
completes the proof of the theorem. \qed
\end{proof}

The following corollary follows immediately.
\begin{coro}
 The $A$-fibered Burnside functor $kB^{n}$ over $p$-groups is uniserial.
\end{coro}
\begin{proof}
Let $K$ and $L$ be subfunctors of $kB^A$ with the respective minimal groups $E$ and $D$. Suppose, without loss of generality,
that the rank $s$ of $E$ is less than the rank $r$ of $D$. Then by the previous theorem, $K$ (resp. L) is generated by 
$e_{E,1}^E$ (resp. $e_{D,1}^D$). We claim that $L\subset K$. To prove this, it is sufficient to show that $L(D) \subset K(D)$. But
as remarked in the proof of the previous theorem, the idempotent $e_{D,1}^D$ is a summand of $\Inf_{E}^D e_{E,1}^E$ and since
$K(D)$ is an ideal of $kB^A(D)$, we have $e_{D,1}^D\in K(D)$. Therefore, by the previous theorem, $L(D)\subset K(D)$. \qed
\end{proof}

Finally we identify the subfunctors and the composition factors of the fibered Burnside functor over $p$-groups. 
Our description is in terms of the well-known subfunctor of intersection kernels, defined as follows. Let $F$ be a fibered biset 
functor and $\mathcal{H}$ be a set of minimal groups of $F$.  Then, the $k$-module $K_{\mathcal{H}}^F(G)$ given by 
\begin{align*}
K_{\mathcal{H}}^F(G)=\bigcap_{\substack{{}_{H}X_{G}\\ H\in \mathcal{H}}}\mbox{\rm ker}\big(X: F(G)\rightarrow F(H)\big)
\end{align*}
together with the induced actions of fibered bisets is a subfunctor of $F$ (cf. \cite[Section 11]{boltjecoskun}).

\begin{prop}
Let $K_{0}$ denote $kB^{A}$. For $i\geq 0$, define $K_{i+1}$ recursively as follows. Let $H_{i}$ be the minimal group of
$K_{i}$, and put $K_{i+1}=K_{\{H_{i}\}}^{K_{i}}$.
Then, 
\begin{enumerate}
\item[\mbox{\rm (i)}] the subfunctor $K_{i+1}$ is the unique maximal subfunctor of $K_{i}$,
\item[\mbox{\rm (ii)}] the minimal group $H_i$ of the subfunctor $K_i$ is the $i$-th element $E_{r_{i}}$ of the set $\mathcal I$ 
defined above.
\end{enumerate}
\end{prop}

\begin{proof}
Note that $K_{i+1}$ is a subfunctor by definition. To see that it is maximal in $K_i$, let $F\subset K_{i}$ be a proper subfunctor. 
We need to show that for any group $G$, we have $F(G)\subseteq K_{i+1}(G)$. To prove this inclusion, it suffices to show that 
$F(H_i)=0$. Indeed, let $x\in F(G)$ be an arbitrary element. Then, for any $A$-fibered 
$(H_i,G)$-biset ${}_{H_i}X_{G}$, we have ${}_{H_i}X_{G}\cdot x \in F(H_i)=0$. It follows that 
$x\in K_{H_i}(G)$ by the definition of $K_{i+1}$. 

Now, note that we have $F(H_i)\subseteq K_{i}(H_i)\cong k\cdot e^{H_i}_{H_i,1}$. In particular, $K_i(H_i)$ is of dimension 1, and since $F$ is proper and 
$K_{i}$ is generated by $e^{H_i}_{H_i,1}$, we must have $F(H_i) = 0$. Indeed, otherwise  $F(H_i) = K_i(H_i)$ and hence $F = K_i$. This shows that $K_{i+1}$ is the unique maximal subfunctor of $K_{i}$, 
completing the proof of the first part. 

For the second part, since the minimal group of $K_0$ is the trivial group $1$ which corresponds to the $0$-th element of the set 
$\mathcal I$, by part (i), it is sufficient to show that if $L\subset F$ are subfunctors of $kB^A$ with $L$ maximal in $F$, and if 
the minimal group of $F$ is $E_{r_i}$, then the minimal group of $L$ is $E_{r_{i+1}}$.

To prove this claim, let the minimal group of $L$ be $E_{r_{j}}$. Since $L$ is a proper subfunctor, by Theorem \ref{generator}, 
we should have $i<j$. Let $K$ denote the subfunctor of $F$ generated by the idempotent $e^{E_{r_{i+1}}}_{E_{r_{i+1}},1}$. Then, 
$K$ is a proper subfunctor of $F$. Indeed, every $A$-fibered $(E_{r_{i}}, E_{r_{i+1}})$-biset decomposes as in Theorem 
\ref{thm:decomposition}. However, the image of $e^{E_{r_{i+1}}}_{E_{r_{i+1}},1}$ under the restriction and the deflation maps are 
zero. Thus, we have $K(E_{r_{i}})=0$. Now, $L$ being maximal guaranties that we have $K(E_{r_{i+1}})\subseteq L(E_{r_{i+1}})$. 
Since $K(E_{r_{i+1}})$ is non-zero, we conclude that $j\leq i+1$ which implies $j=i+1$.  \qed
\end{proof}

Thus we have shown that for each $r_{i}\in \mathcal{I}$, there is a subfunctor, namely $K_{i}$, of $kB^{A}$. Next we examine the 
set $\mathcal{I}=\{0\}\cup \{r\in \mathbb{N}\mid p^{r-1}\equiv 1~(\text{mod}~q)\}$ more closely. 

If $q=0$, then clearly we have $\mathcal{I}=\{0,1\}$. If $q\neq 0$, then $\mathcal{I}$ consists of all positive integers congruent to 
$1$ modulo $s$ where $s$ is the order of $p$ modulo $q$. Note further that if $q$ divides $p-1$, then the order $s$ is equal to 
1 and $\mathcal{I}$ consists of all positive integers. Now we are ready to state our main theorem.

\begin{theorem}\label{lasttheorem}
Let $A$ be a cyclic $p$-group and $k$ be a sufficiently large field of characteristic $q$ with $q\neq p$. Then, the $A$-fibered 
Burnside functor $kB^{A}$ over $p$-groups is uniserial. Moreover we have
\begin{align*}
&kB^{A}=K_{0}\supset K_{1}\supset K_{\infty}=\{0\} && \text{if}~ q=0,\\
&kB^{A}=K_{0}\supset K_{1}\supset K_{2}\supset K_{3}\supset\cdots && \text{if} ~q\neq 0 ~\text{and} ~q\mid p-1,\\
&kB^{A}=K_{0}\supset K_{1}\supset K_{1+s}\supset K_{1+2s}\supset\cdots && \text{if} ~q\neq 0 ~\text{and} ~q\nmid p-1.
\end{align*}
where the subfunctors $K_i$ are as defined above and for each $i$, the simple quotient $K_{i}/K_{i+1}$ is isomorphic to the
simple $A$-fibered $p$-biset functor $S_{E_{r_{i}},1}$. 
\end{theorem}
\begin{proof}
All the parts of the theorem is proved except the last claim concerning the simple composition factors.  Note that the quotient $S_i = K_i/K_{i+1}$ is 
simple and the minimal group of the quotient is $E_{r_i}$. Thus it is sufficient to show that for each $i$, the $k$-vector 
space $S_i(E_{r_{i}})$ is the trivial $k[E_{r_{i}}^{*}\rtimes \Out(E_{r_{i}})]$-module.  

However, we have $S_i (E_{r_{i}})=k\cdot e^{E_{r_{i}}}_{E_{r_{i}},1}$. Hence it suffices to show, for 
$\phi \in E_{r_{i}}^{*}$ and $\lambda \in \Out(E_{r_{i}})$, that the effects of the fibered bisets $\tw^{\phi}_{E_{r_{i}},E_{r_{i}}}$ and 
$c^{\lambda}_{E_{r_{i}},E_{r_{i}}}$ on the idempotent $e^{E_{r_{i}}}_{E_{r_{i}},1}$ are trivial. But by Proposition \ref{pro:iso-tw}, we have 
\begin{align*}
\tw^{\phi}_{E_{r_{i}},E_{r_{i}}}\cdot e^{E_{r_{i}}}_{E_{r_{i}},1}= \phi(1)\cdot e^{E_{r_{i}}}_{E_{r_{i}},1} =e^{E_{r_{i}}}_{E_{r_{i}},1}
\end{align*}
and 
\begin{align*}
c^{\lambda}_{E_{r_{i}},E_{r_{i}}}\cdot e^{E_{r_{i}}}_{E_{r_{i}},1}= e^{E_{r_{i}}}_{\lambda(E_{r_{i}}),\lambda(1)}=e^{E_{r_{i}}}_{E_{r_{i}},1}
\end{align*}
as required.  \qed 
\end{proof}
\begin{rem}
As remarked above, a simple fibered biset functor may have two non-isomorphic minimal groups. However, this is not the case for the simple functors
that appear in the previous theorem. Indeed, we already know that any minimal group for a subfunctor of the fibered Burnside functor must be an
elementary abelian $p$-group and, of a given order, there is a unique elementary abelian $p$-group.
\end{rem}
\begin{rem}
Let $k$ be a field of characteristic zero and $A$ be a non-trivial cyclic $p$-group. Then, by the above 
theorem, there is a short exact sequence 
\begin{displaymath}
\begin{diagram}
0&\rTo& S_{C_p,1} &\rTo &kB^A&\rTo& S_{1,1}&\rTo&0
\end{diagram}
\end{displaymath}
of $A$-fibered $p$-biset functors. One can show that the simple head can be identified with the functor $kR_A$ of $A$-monomial characters and the quotient map can be chosen as the linearization map. Hence the above sequence becomes
\begin{displaymath}
\begin{diagram}
0&\rTo& S_{C_p,1} &\rTo &kB^A&\rTo& kR_A&\rTo&0
\end{diagram}
\end{displaymath}

Recall from \cite{boucthevenaz} that, in the case of $p$-biset functors, that is, when $A$ is trivial, the 
corresponding sequence is
\begin{displaymath}
\begin{diagram}
0&\rTo& kD &\rTo &kB&\rTo& kR_{\mathbb Q}&\rTo&0
\end{diagram}
\end{displaymath}
where $k D$ is the functor of torsion-free part of the Dade group, $B$ is the (ordinary) Burnside functor and $k R_{\mathbb Q}$ is the 
functor of rational representations. Existence of this sequence is one of the key results in the classification of endo-permutation modules.
We do not know any natural construction that would match the simple $A$-fibered $p$-biset functor 
seen in the above short exact sequence.
\end{rem}


\begin{thebibliography}{9}

\bibitem{barker} Barker L., ``Fibred permutation sets and the idempotents and units of monomial Burnside rings'', \textit{J. Algebra} 281 (2004) 535-566.

\bibitem{boltjecanonical} Boltje R., ``A general theory of canonical induction formulae'', \textit{J. Algebra} 206 (1998) 293-343.

\bibitem{boltje} Boltje R., ``Representation rings of finite groups, their species and idempotent formulae'', \textit{to appear in J. Algebra}.

\bibitem{boltjecoskun} Boltje R., Co\c{s}kun O.,   ``Fibered Biset Functors'', \textit{arXiv:1612.01117 }.

\bibitem{boucdecomposition} Bouc S., ``Foncteurs d'ensembles munis d'une double action'', \textit{J. Algebra} 183 (1996) 664-736.


\bibitem{boucthevenaz} Bouc S., Thev\'{e}naz J., ``The group of endo-permutation modules'', \textit{Invent. Math} 139 (2000) 
275-349.

\bibitem{BDade}
Bouc S., `The Dade group of a $p$-group',  {\em Invent. Math.} 164 (2006), 189-231. 


\bibitem{bouc} Bouc S., Biset functors for finite groups, Lecture Notes in Math., vol.1990 Springer-Verlag, Berlin (2010).


\bibitem{dressmackey} Dress A.W.M., ``Contributions to theory of induced representations'', \textit{Lecture Notes in Math.} 342, Springer-Verlag, New York, 183-240 (1973).


\bibitem{dress} Dress A.W.M., ``The ring of monomial representations, I. Structure theory'', \textit{J. Algebra} 18 (1971) 137-157.

\bibitem{green} Green J.A., ``Axiomatic representation theory for finite groups'', \textit{J. Pure and Appl. Algebra} 1 (1971) 41-771.

\bibitem{R} Romero N,``On fibered biset functors with fibres order of prime and four", \textit{J. Algebra} 387 (2013), 185?194.

\bibitem{TW}
{Thévenaz J., Webb P.}, `The structure of Mackey functors', {\em Trans. Amer. Math. Soc. } 347 (1995), 1865-1961. 


\end{thebibliography}
\end{document}